\documentclass{amsart}% option [reqno] puts equat numb to the right

\usepackage[mathscr]{eucal}% So-called Euler fonts, allows \mathscr
\usepackage{amssymb}% allows special arrows like >-> e.g.
\usepackage[usenames,dvipsnames]{color}
\usepackage{amsthm}% allows theorem* , e.g.
\usepackage{bbold}% allows \unit \mathbb{1}
\usepackage{enumerate}% allows easy change of label
\usepackage{array}% allows array
\usepackage{amsmath}% allows pmatrix
\numberwithin{equation}{section}% makes equat numb contain the section
\usepackage[all]{xy}
\SelectTips{cm}{}
\newdir{ >}{{}*!/-10pt/\dir{>}}

\usepackage[colorlinks=true,linkcolor={Brown},citecolor={Brown},urlcolor={Brown}]{hyperref}

\usepackage{cleveref}% allows references via \cref and \Cref which automatically repeat the name (Thm, Prop, etc) of the ref.

\swapnumbers %pour que le numéro apparaisse devant le théorème

\newtheorem{Cor}[equation]{Corollary}
\newtheorem{Lem}[equation]{Lemma}
\newtheorem{Prop}[equation]{Proposition}
\newtheorem{Thm}[equation]{Theorem}

\theoremstyle{remark}
\newtheorem{Def}[equation]{Definition}
\newtheorem{Not}[equation]{Notation}
\newtheorem{Exa}[equation]{Example}
\newtheorem{Exas}[equation]{Examples}
\newtheorem{Hyp}[equation]{Hypothesis}
\newtheorem{Rem}[equation]{Remark}

\newcommand{\nc}{\newcommand}
\nc{\dmo}{\DeclareMathOperator}

\dmo{\Ab}{Ab}
\dmo{\Add}{Add}
\dmo{\coker}{coker}
\dmo{\Der}{D}% ground notation for derived categories
\dmo{\Grmod}{Mod_{\sbull}}
\dmo{\Hom}{Hom}
\dmo{\Id}{Id}
\dmo{\im}{im}
\dmo{\Ind}{Ind}
\dmo{\Ker}{Ker}
\dmo{\modname}{mod}%
\dmo{\Mod}{Mod}% sheaves of modules
\dmo{\opname}{op}
\dmo{\Proj}{Proj}%
\dmo{\Qcoh}{Qcoh}% quasi-coherent modules
\dmo{\SH}{SH}% ground name for cat of spectra
\dmo{\smallh}{h}%
\dmo{\smallperf}{perf}% ground exponent for ``perfect''
\dmo{\Spc}{Spc}
\dmo{\Spec}{Spec}
\dmo{\stab}{stab}% stable category of fin. gen. mod.
\dmo{\Stab}{Stab}% stable category of non-fin. gen. mod.
\dmo{\supp}{supp}
\dmo{\fpname}{fp}
\dmo{\yoneda}{h}

\nc{\ababs}{{\sl ab absurdo}}
\nc{\AKB}{\bat{A}(\cat{K};\cat{B})}%
\nc{\ala}{{\sl \`a la}\ }
\nc{\bat}[1]{\bar{\cat{#1}}}
\nc{\bbF}{\mathbb{F}}
\nc{\bbQ}{\mathbb{Q}}
\nc{\boneda}{\bar\yoneda}
\nc{\calF}{\mathcal{F}}
\nc{\calH}{\mathcal{H}}
\nc{\calO}{\mathcal{O}}
\nc{\cat}[1]{\mathscr{#1}}%or: \nc{\cat}[1]{\mathcal{#1}}
\nc{\cA}{\cat{A}}
\nc{\cB}{\cat{B}}
\nc{\cC}{\cat{C}}
\nc{\cJ}{\cat{J}}
\nc{\cK}{\cat{K}}
\nc{\cL}{\cat{L}}
\nc{\colim}{\mathop{\mathrm{colim}}}
\nc{\cP}{\cat{P}}
\nc{\cQ}{\cat{Q}}
\nc{\cT}{\cat{T}}
\nc{\Displ}{\displaystyle}
\nc{\Db}{\Der^{\mathrm{b}}}
\nc{\Dperf}{\Der^{\smallperf}}% derived category of perfect compl
\nc{\EB}{E_{\cB}}% residue weak ring at B
\nc{\EBp}{E_{\cB'}}% residue weak ring at B
\nc{\eg}{{\sl e.g.}}
\nc{\fp}{^{\fpname}}
\nc{\gp}{\mathfrak{p}}% prime p
\nc{\GGrmod}{\textrm{-}\Grmod}% left graded modules
\nc{\hB}{\boneda_{\cB}}% residue functor at B
\nc{\hEB}{\hat E_{\cB}}% residue weak ring at B in MT
\nc{\hEBp}{\hat E_{\cB'}}% residue weak ring at B in MT
\nc{\Homcat}[1]{\Hom_{\cat #1}}
\nc{\hook}{\hookrightarrow}
\nc{\ie}{{\sl i.e.}\ }
\nc{\into}{\mathop{\rightarrowtail}}
\nc{\inv}{^{-1}}
\nc{\kk}{k}%{\Bbbk}
\nc{\Locab}[1]{\langle#1\rangle}% \Loc({#1})}
\nc{\Locat}[1]{\Locab{\cat{#1}}}% \Loc(\cat{#1})}
\nc{\loccit}{{\sl loc.\ cit.}}
\nc{\Mid}{\,\big|\,}
\nc{\mK}{\mmod\cat{K}}% most used
\nc{\MK}{\MMod\cat{K}}% most used
\nc{\mmod}{\modname\kern-0.1em\text{-}}%
\nc{\MMod}{\Mod\kern-0.1em\text{-}}%
\nc{\onto}{\mathop{\twoheadrightarrow}}
\nc{\op}{^{\opname}}
\nc{\oursetminus}{\smallsetminus}
\nc{\potimes}[1]{^{\otimes #1}}% tensor power
\nc{\sbull}{{\scriptscriptstyle\bullet}}
\nc{\SET}[2]{\big\{\,#1\Mid#2\,\big\}}
\nc{\Spch}{\Spc^{\smallh}}
\nc{\SpchK}{\Spch(\cat{K})}% most used
\nc{\SpcK}{\Spc(\cat{K})}% most used
\nc{\sstab}{\,\text{--}\stabname}%
\nc{\supph}{\supp^{\smallh}}
\nc{\tensid}{$\otimes$-ideal}
\nc{\too}{\mathop{\longrightarrow}\limits}
\nc{\unit}{\mathbb{1}}% unit for \otimes
\nc{\vcorrect}[1]{{\vphantom{\vbox to #1em{}}}}

\begin{document}

%------------------------------------------------------------------------------

\title[nilpotence theorems via homological residue fields]{nilpotence theorems \\ via homological residue fields}
\author{Paul Balmer}
\date{2019 January 22}

\address{Paul Balmer, Mathematics Department, UCLA, Los Angeles, CA 90095-1555, USA}
\email{balmer@math.ucla.edu}
\urladdr{http://www.math.ucla.edu/$\sim$balmer}

\begin{abstract}
We prove nilpotence theorems in tensor-triangulated categories using suitable Gabriel quotients of the module category, and discuss examples.
\end{abstract}


\thanks{Author supported by NSF grant~DMS-1600032.}

\maketitle

%------------------------------------------------------------------------------
%------------------------------------------------------------------------------

\section{Introduction}

%------------------------------------------------------------------------------

For the average Joe, and the median Jane, \emph{the} Nilpotence Theorem refers to a result in stable homotopy theory, conjectured by Ravenel and proved by Devinatz, Hopkins and Smith in their famous work on chromatic theory~\cite{DevinatzHopkinsSmith88,HopkinsSmith98}. One form of the result says that a map between finite spectra which gets annihilated by all Morava $K$-theories must be tensor-nilpotent. Under Hopkins's impetus~\cite{Hopkins87}, these ideas soon expanded beyond topology. Neeman~\cite{Neeman92a} in commutative algebra and Thomason~\cite{Thomason97} in algebraic geometry proved nilpotence theorems for maps in derived categories of perfect complexes, using ordinary residue fields instead of Morava $K$-theories. Benson, Carlson and Rickard~\cite{BensonCarlsonRickard97} led the charge into yet another area, namely modular representation theory of finite groups, where the appropriate `residue fields' turned out to be shifted cyclic subgroups, and later $\pi$-points~\cite{FriedlanderPevtsova07}. As further areas kept joining the fray, expectations rose of a unified treatment applicable to every tensor-triangulated category in Nature. In this vein, Mathew~\cite[Thm.\,4.14\,(b)]{Mathew17a} proved an abstract nilpotence theorem via $E_\infty$-rings in $\infty$-categories over the field~$\mathbb{Q}$. However, this rationality assumption is a severe restriction, incompatible with the chromatic joys of topological Joe and the positive characteristic tastes of modular Jane. Here, we prove abstract nilpotence theorems, integrally and without $\infty$-categories. For instance, \Cref{Cor:flagship} says:
\begin{Thm}
\label{Thm:nil-compact}%
Let $f\colon x\to y$ be a morphism in an essentially small, rigid tensor-triangulated category~$\cat{K}$. If we have $\boneda(f)=0$ for every homological residue field $\boneda\colon \cat{K}\to\bat{A}$, then there exists $n\ge 1$ such that $f\potimes{n}=0$ in~$\cat{K}$.
\end{Thm}

We need to explain the \emph{homological residue fields} that appear in this statement. Their purpose is to encapsulate the key features of Morava $K$-theories, ordinary residue fields, shifted cyclic subgroups, etc, from an abstract point of view. In other words, when first meeting a tensor-triangulated category~$\cat{K}$, we would like to extract its `residue fields' without knowing intimate details about~$\cat{K}$, as we are used to extract residue fields $\kappa(\gp)=R_\gp/\gp R_\gp$ from any commutative ring~$R$, without much knowledge about~$R$ beyond its propensity to harbor prime ideals $\gp\in\Spec(R)$. We investigated this question of `tensor-triangular fields' in the recent joint work~\cite{BalmerKrauseStevenson19}, with Krause and Stevenson. Although the present article can be read independently, we refer to that prequel for motivation, background, justification, and a couple of lemmas. In retrospect, our nilpotence theorems further validate the ideas introduced in~\cite{BalmerKrauseStevenson19}.

In a nutshell, as we do not know how to produce residue fields within triangulated categories, we consider instead homological tensor-functors to \emph{abelian} categories
\[
\boneda=\boneda_{\cB}\colon \ \cat{K}\ \hook \ \mK \ \onto \ (\mK)/\cat{B}
\]
composed of the Yoneda embedding $\yoneda\colon \cat{K}\hook \mK$ into the Freyd envelope of~$\cK$ (\Cref{Rem:Freyd}) followed by the Gabriel quotient $Q_{\cB}\colon \mK\onto (\mK)/\cat{B}$ with respect to any \emph{maximal} $\otimes$-ideal Serre subcategory~$\cat{B}$. Thus \Cref{Thm:nil-compact} can be rephrased as follows:

\smallbreak
\textit{If a morphism $f\colon x\to y$ in~$\cK$ is annihilated by $\boneda_{\cB}\colon\cat{K}\to (\mK)/\cB$ for every maximal Serre $\otimes$-ideal $\cB\subset\mK$, then $f$ is $\otimes$-nilpotent in~$\cK$.}
\smallbreak

At first, it might be counter-intuitive to only invoke \emph{maximal} $\otimes$-ideals~$\cB$, instead of some kind of more general `\emph{prime}' $\otimes$-ideals of $\mK$, but we explain in \Cref{sec:homol-fields} why this notion covers all points of the triangular spectrum~$\SpcK$ of~$\cat{K}$, not just the closed points. We also explain in \Cref{Rem:fields-vs-fields} how the above homological residue fields correspond to the local constructions proposed in~\cite[\S\,4]{BalmerKrauseStevenson19}.

As a matter of fact, in the examples, there exist alternative formulations of the nilpotence theorem. And the same holds here. Most notably, if our triangulated category $\cat{K}$ sits inside a `big' one, $\cat{K}\subset\cat{T}$, as the compact objects $\cat{K}=\cat{T}^c$ of a so-called `rigidly-compactly generated' tensor-triangulated category~$\cat{T}$ (\Cref{Rem:big-T}), we expect a nilpotence theorem for maps $f\colon x\to Y$ with compact source $x\in\cat{K}$ but \emph{arbitrary} target~$Y$ in~$\cat{T}$. This flavor of nilpotence theorem is \Cref{Cor:nil-big}.

In order to handle such generalizations, we consider the big Grothen\-dieck category $\cat{A}:=\MK$ of \emph{all} right $\cat{K}$-modules (\Cref{Not:A}), not just the subcategory of finitely presented ones that is the Freyd envelope $\cat{A}\fp=\mK$. When~$\cK=\cT^c$, the big category~$\cat{T}$ still admits a so-called `restricted-Yoneda' functor $\yoneda\colon\cT\to \MK$ (\Cref{Rem:big-T}). Every maximal Serre $\otimes$-ideals $\cB\subset\cat{A}\fp$ generates a localizing (Serre) $\otimes$-ideal $\Locat{B}$ of~$\cat{A}$ and we can consider the corresponding `big' Gabriel quotient $\bat{A}:=\cat{A}/\Locat{B}$. Composing with restricted-Yoneda, we obtain a homological $\otimes$-functor $\boneda_{\cB}\colon \cat{T} \too \bat{A}$ on the `big' category~$\cT$, extending the one on~$\cat{K}$\,:
\[
\xymatrix@R=1em{
\cK =\cT^c \ \ar@{^(->}^-{\yoneda}[r] \ar@{}[d]|-{\bigcap}
 \ar@{-<} `u/1em[r]`/1em[rr]^-{\displaystyle\boneda_{\cB}} [rr]
& \cat{A}\fp=\mK \ar@{->>}[r]^-{Q_{\cB}} \ar@{}[d]|-{\bigcap}
& \bat{A}\fp=(\mK)/\cB \ar@{}[d]|-{\bigcap}
\\
\cT \ar[r]^-{\yoneda}
 \ar@{-<} `d/1em[r]`/1em[rr]_-{\displaystyle\boneda_{\cB}} [rr]
& \cat{A}=\MK \ar@{->>}[r]^-{Q_{\cB}}
& \bat{A}=(\MK)/\Locat{B} \,.\!\!
}
\]
Thanks to~\cite[Prop.\,A.14]{BalmerKrauseStevenson17app}, the image $\yoneda(Y)$ of every object $Y$ in the big category~$\cT$ remains \emph{$\otimes$-flat} in the module category~$\cA=\MK$, meaning that the functor $\yoneda(Y)\otimes-\colon\cA\to \cA$ is exact. In fact, this $\otimes$-flatness plays an important role in the proof of the nilpotence theorem. In particular, \Cref{Cor:main} tells us:
\begin{Thm}
\label{Thm:nil-flat-intro}%
Let $f\colon \yoneda(x)\to F$ be a morphism in $\cA=\MK$, for $x\in\cat{K}$. Suppose that the $\cK$-module $F$ is $\otimes$-\emph{flat} and that $Q_\cB(f)=0$ in $\bat{A}=\cat{A}/\Locat{B}$, for every maximal Serre $\otimes$-ideal~$\cB\subset\mK$. Then $f$ is $\otimes$-nilpotent in~$\MK$.
\end{Thm}

All these statements are corollaries of our most general Nilpotence Theorem~\ref{Thm:main}, which further involves a `parameter' \ala Thomason~\cite{Thomason97}, \ie a closed subset $W\subseteq\SpcK$ of the spectrum on which we test the vanishing of~$f$.

\bigbreak

Finally, in \Cref{sec:examples}, we classify those homological residue fields in examples. For brevity, let us pack three theorems into one:
\begin{Thm}
\label{Thm:examples}%
There exists a surjection $\phi\colon \Spch(\cK)\onto\SpcK$ from the set of maximal Serre $\otimes$-ideals~$\cB\subset\mK$ to the triangular spectrum of~$\cK$. Moreover, it is a bijection for each of the following tensor-triangulated categories~$\cK$:
\begin{enumerate}[\rm(a)]
\item
\label{it:intro-Dperf}%
Let $X$ be a quasi-compact and quasi-separated scheme and $\cK=\Dperf(X)$ its derived category of perfect complexes. (\Cref{Cor:Dperf}.)
\item
\label{it:intro-SH}%
Let $G$ be a compact Lie group and $\cK=\SH(G)^c$ the $G$-equivariant stable homotopy category of finite genuine $G$-spectra. In particular, this holds for $\cK=\SH^c$ the usual stable homotopy category. (\Cref{Cor:SH}.)
\item
\label{it:intro-stab}%
Let $G$ be a finite group scheme over a field~$k$ and $\cK=\stab(kG)$ its stable module category of finite-dimensional $kG$-modules modulo projectives. (\Cref{Exa:Stab}.)
\end{enumerate}
\end{Thm}

We caution that the above does \emph{not} give a new proof of the nilpotence theorems known in those examples, except perhaps for modular representation theory~\eqref{it:intro-stab}, as we discuss further in \Cref{Rem:stab-nil}. Indeed, the above results rely on existing classification results, which themselves often rely on a form of nilpotence theorem. These results should rather be read as a converse to our our nilpotence theorems via homological residue fields: If a collection of homological residue fields detects nilpotence then that collection contains all homological residue fields (\Cref{Thm:converse2nil}).

\bigbreak
\noindent\textbf{Acknowledgments}: I am very thankful to Henning Krause and Greg Stevenson for their precious support, during my stay at Bielefeld University and later during the preparation of this paper. I also thank Beren Sanders for useful discussions.

%------------------------------------------------------------------------------
\bigbreak\goodbreak
\section{Background and notation}
\label{sec:basics}%
\medbreak
%------------------------------------------------------------------------------

%
\begin{Hyp}
\label{Hyp:0}%
Throughout the paper, we denote by $\cat{K}$ an essentially small tensor-triangulated category and by $\unit\in\cat{K}$ its $\otimes$-unit. We often assume $\cat{K}$ \emph{rigid}, in the sense recalled in \Cref{Rem:rigid} below. See details in~\cite[\S\,1]{Balmer05a} or~\cite[\S\,1]{Balmer10b}.
\end{Hyp}

\begin{Exas}
Such~$\cat{K}$ include the usual suspects: in topology $\cat{K}=\SH^c$ the stable homotopy category of finite spectra; in algebraic geometry $\cat{K}=\Dperf(X)=\Der_{\Qcoh}(\calO_X\textrm{-Mod})^c$ the derived category of perfect complexes over a scheme~$X$ which is assumed quasi-compact and quasi-separated (\eg\ a noetherian, or an affine one); in modular representation theory $\cat{K}=\stab(\kk G)=\Stab(\kk G)^c$ the stable category of finite-dimensional $\kk$-linear representations of a finite group~$G$ over a field~$\kk$ of characteristic dividing the order of~$G$. But there are many more examples: equivariant versions, categories of motives, $KK$-categories of $C^*$-algebras, etc, etc.
\end{Exas}

\begin{Rem}
\label{Rem:tt}%
We use the \emph{triangular spectrum} $\SpcK=\SET{\cP\subset \cK}{\cP\textrm{ is a prime}}$ where a proper thick $\otimes$-ideal $\cP\subsetneq\cK$ is called a \emph{(triangular) prime} if $x\otimes y\in\cP$ implies $x\in\cP$ or $y\in\cP$. The \emph{support} of an object $x\in\cK$ is the closed subset $\supp(x):=\SET{\cP\in\SpcK}{x\notin \cP}$ $=$ $\SET{\cP\in\SpcK}{x\textrm{ is non-zero in }\cK/\cP}$. These are exactly the so-called \emph{Thomason closed} subsets of~$\SpcK$, \ie those closed subsets $Z\subseteq\SpcK$ with  quasi-compact open complement $\SpcK\oursetminus Z$, by~\cite[Prop.\,2.14]{Balmer05a}.
\end{Rem}

\begin{Rem}
\label{Rem:rigid}%
Rigidity of $\cat{K}$ will play an important role in the proof of the main \Cref{Thm:main}. Rigidity means that every object $x\in\cat{K}$ is strongly-dualizable, hence admits a dual $x^\vee\in\cat{K}$ with an adjunction $\Homcat{K}(x\otimes y,z)\cong \Homcat{K}(y,x^\vee\otimes z)$. In particular, the unit-counit relation forces $x\otimes\eta_x\colon x\to x\otimes x^\vee\otimes x$ to be a split monomorphism where $\eta_x\colon \unit\to x^\vee\otimes x$ is the unit of the adjunction. This implies that $x$ is a direct summand of a $\otimes$-multiple of $x\potimes{n}$ for any~$n\ge1$. It follows that if a map $f$ satisfies $(f\otimes x)\potimes{n}=0$ then $f\potimes{n}\otimes x=0$ as well.
\end{Rem}

\begin{Not}
\label{Not:A}%
The Grothendieck abelian category $\MK$ of right $\cat{K}$-modules, \ie additive functors $M\colon \cat{K}\op\to \Ab$, receives $\cat{K}$ via the Yoneda embedding, denoted
\[
\begin{array}{rcl}
\yoneda\,:\quad \cat{K} & \hook & \MK=\Add(\cat{K}\op,\Ab)
\\
x & \mapsto & \hat {x} :=\Homcat{K}(-,x)
\\
f & \mapsto & \hat {f}\,.
\end{array}
\]
\end{Not}

\begin{Rem}
\label{Rem:Freyd}%
Let us recall some standard facts about $\cat{K}$-modules. See details in~\cite[App.\,A]{BalmerKrauseStevenson17app}. By Day convolution, the category $\cat{A}=\MK$ admits a tensor $\otimes\colon \cat{A}\times\cat{A}\to\cat{A}$ which is colimit-preserving (in particular \emph{right-exact}) in each variable and which makes $\yoneda\colon \cat{K}\to \cat{A}$ a tensor functor: $\widehat{x\otimes y}\cong\hat{x}\otimes\hat{y}$. Hence $\yoneda$ preserves rigidity, so $\hat{x}$ will be rigid in~$\cat{A}$ whenever~$x$ is in~$\cat{K}$. Moreover, the object $\hat{x}\in\cat{A}$ is finitely presented projective and $\otimes$-flat in~$\cat{A}$. The tensor subcategory
\[
\cat{A}\fp=\mK\subset\MK=\cat{A}
\]
of finitely presented objects is itself abelian and is nothing but the classical \emph{Freyd envelope} of~$\cat{K}$, see~\cite[Chap.\,5]{Neeman01}. Recall that $\yoneda\colon \cK\hook \mK$ is the universal homological functor out of~$\cat{K}$, and that a functor from a triangulated category to an abelian category is \emph{homological} if it maps exact triangles to exact sequences. Every object of~$\cat{A}$ is a filtered colimit of finitely presented ones. In short, $\cat{A}$ is a \emph{locally coherent} Grothendieck category.
\end{Rem}

\begin{Rem}
\label{Rem:B}%
Given a Serre subcategory $\cB\subseteq\cat{A}\fp$ we can form $\Locat{B}$, or~$\overrightarrow{\cat{B}}$, the localizing subcategory of~$\cat{A}$ generated by~$\cat{B}$. The subcategory~$\Locat{B}$ is the smallest Serre subcategory containing~$\cB$ and closed under coproducts; it consists of all (filtered) colimits in~$\cat{A}$ of objects of~$\cat{B}$. For instance $\Locab{\cat{A}\fp}=\cat{A}$ and it follows that if $\cB$ is $\otimes$-ideal in~$\cat{A}\fp$ then so is $\Locat{B}$ in~$\cat{A}$. We denote the corresponding Gabriel quotient~\cite{Gabriel62} by
\[
Q_\cB\colon \cat{A}\onto \cat{A}/\Locat{B}.
\]
We recall that $\Locat{B}$ is also locally coherent with $\Locat{B}\fp=\cat{B}$ and so is the quotient $\bat{A}$ with $\bat{A}\fp\cong \cat{A}\fp/\cat{B}$. When $\cB$ is $\otimes$-ideal then $\bat{A}$ inherits a unique tensor structure turning $Q_{\cB}\colon\cat{A}\onto \bat{A}$ into a tensor functor, which preserves $\otimes$-flat objects. All this remains true without assuming $\cat{K}$ rigid. See details in~\cite[\S\,2]{BalmerKrauseStevenson19}.
\end{Rem}

\begin{Rem}
\label{Rem:adj-potimes}%
For $\cat{K}$ rigid, consider the special case of the adjunction $\Homcat{K}(x,y)\cong \Homcat{K}(\unit,x^\vee\otimes y)$. Under this isomorphism, if $f\colon x\to y$ corresponds to $g\colon \unit\to x^\vee \otimes y$ then for $n\ge1$ the morphism $f\potimes{n}\colon x\potimes{n}\to y\potimes{n}$ corresponds to \mbox{$g\potimes{n}\colon \unit\to (x^\vee\otimes y)\potimes{n}$} under the analogous isomorphism $\Homcat{K}(x\potimes{n},y\potimes{n})\cong \Homcat{K}(\unit,(x\potimes{n})^\vee\otimes y\potimes{n})$ $\cong \Homcat{K}(\unit,(x^\vee\otimes y)\potimes{n})$. In particular, $f$ is $\otimes$-nilpotent if and only if~$g$ is. Note that the above observation only uses that $x$ is rigid in a tensor category and does not use that $y$ itself is rigid. We can therefore also use this trick for any morphism $f\colon\hat{x}\to M$ in the module category~$\cat{A}=\MK$, as long as $x$ comes from~$\cK$.
\end{Rem}

We shall need the following folklore result about modules and localization:

\begin{Prop}
\label{Prop:Mod-quotients}%
Let $\cat{J}\subset\cat{K}$ be a thick $\otimes$-ideal and let~$q\colon \cat{K}\to \cat{L}$ be the corresponding Verdier quotient $\cat{K}\onto \cat{K}/\cat{J}$ or its idempotent completion $\cat{K}\to (\cat{K}/\cat{J})^\natural$. Consider the left Kan extension $q_!\colon \MK\too \MMod\cat{L}$, left-adjoint to restriction $q^*\colon \MMod\cat{L}\too \MK$ along~$q$. Then $q_!$ is a localization identifying $\MMod\cat{L}$ as the Gabriel quotient of~$\MK$ by $\Ker(q_!)=\Locab{\yoneda(\cat{J})}$ the localizing subcategory generated by~$\yoneda(\cat{J})$. The localization~$q_!$ restricts to a localization $q_!\colon \mK\onto \mmod\cat{L}$ on finitely presented objects, identifying $\mmod\cat{L}$ as the quotient of~$\mK$ by $\Ker(q_!)\fp$ which is also the Serre envelope of~$\yoneda(\cat{J})$ in~$\mK$.
\end{Prop}

\begin{proof}
The fact that $q_!$ is a localization follows from~\cite[Thm.\,4.4 and \S\,3]{Krause05}. The left Kan extension $q_!(M)$ is defined as $\colim_{\alpha\colon\hat{x}\to M}q_!(\hat{x})$ with $q_!(\hat{x})=\widehat{q(x)}$. To identify the kernel of~$q_!\colon \MK\to \MMod\cat{L}$, since ${\yoneda(\cat{J})}\subseteq \Ker(q_!)\fp$ is clear, it suffices to show $\Ker(q_!)\subseteq \Locab{\yoneda(\cat{J})}$. For every $M\in\Ker(q_!)$, using that $\widehat{q(x)}$ is finitely presented projective for all~$x\in\cat{K}$ together with faithfulness of Yoneda $\cat{L}\hook \MMod\cat{L}$, one shows that every morphism $\alpha\colon\hat{x}\to M$ with $x\in\cat{K}$ factors via a morphism $\hat{\beta}\colon \hat{x}\to \hat{y}$ where $q(\beta)$ vanishes in~$\cat{K}/\cat{J}$, meaning that the morphism $\beta\colon x\to y$ in~$\cat{K}$ factors via an object of~$\cat{J}$. In short, every morphism $\hat{x}\to M$ factors via an object in~$\yoneda(\cat{J})$ which implies that~$M$ belongs to the localizing subcategory~$\Locab{\yoneda(\cat{J})}$. The $\otimes$-properties are then easily added onto this purely abelian picture.
\end{proof}

%------------------------------------------------------------------------------
\bigbreak\goodbreak
\section{Homological primes and homological residue fields}
\label{sec:homol-fields}%
\medbreak
%------------------------------------------------------------------------------

Let $\cat{K}$ be a tensor-triangulated category as in \Cref{Hyp:0}.

\begin{Def}
\label{Def:homol-fields}%
A \emph{(coherent) homological prime} for~$\cK$ is a \emph{maximal} proper Serre $\otimes$-ideal subcategory $\cat{B}\subset\cat{A}\fp=\mK$ of the Freyd envelope of~$\cat{K}$. The \emph{homological residue field} corresponding to~$\cB$ is the functor constructed as follows
\[
\xymatrix@C=1em@R=.2em{\boneda_{\cB}=Q_{\cat{B}}\circ \yoneda\colon
& \cat{K} \ \ar@{^(->}[rr]^-{\yoneda}
&& \cat{A}=\MK \ar@{->>}[rr]^-{Q_{\cat{B}}}
&& \bat{A}(\cat{K};\cat{B}):=\frac{\Displ\MK}{\Displ\Locat{B}}
\\
& x \ \ar@{|->}[rr]
&& \quad \hat{x} \quad \ar@{|->}[rr]
&& \quad \bar{x}\,. \kern8em
}
\]
The functor $\boneda_{\cB}\colon \cK\to \AKB\fp$ is a (strong) monoidal homological functor (\Cref{Rem:Freyd}), that lands in the finitely presented subcategory~$\AKB\fp\cong(\mK)/\cat{B}$. By construction, the tensor-abelian category $\AKB\fp$ has only the trivial Serre $\otimes$-ideals,~$0$ and $\AKB\fp$ itself. These homological residue fields are truly the same as the homological $\otimes$-functors constructed in~\cite[\S\,4]{BalmerKrauseStevenson19}, up to a little paradigm change that we explain in \Cref{Rem:fields-vs-fields} below.
\end{Def}

\begin{Rem}
Since $\cat{K}$ is essentially small, its Freyd envelope,~$\mK$, admits only a \emph{set} of Serre subcategories. So we can apply Zorn to construct homological primes and homological residue fields as soon as $\cat{K}\neq0$. Contrary to what happens with commutative rings, these \emph{maximal} Serre $\otimes$-ideals are not only picking up `closed points' as one could first fear. In fact, they `live' above \emph{every} prime of the triangular spectrum $\SpcK$ of~$\cat{K}$ (\Cref{Rem:tt}). First, let us explain the relationship.
\end{Rem}

\begin{Prop}
\label{Prop:field-to-prime}%
Let $\cB$ be a homological prime with homological residue field $\boneda_{\cB}\colon \cat{K}\to \bat{A}(\cat{K};\cat{B})$. Then $\cP(\cB):=\Ker(\boneda_{\cB})=\yoneda\inv(\cat{B})$ is a triangular prime of~$\cat{K}$.
\end{Prop}

\begin{proof}
Since Yoneda~$\yoneda\colon\cat{K}\to \mK$ is homological and (strong) monoidal, the preimage $\cP(\cB)=\yoneda\inv(\cB)$ is a proper thick $\otimes$-ideal of~$\cat{K}$. To see that $\cP(\cB)$ is prime, let $x,y\in\cat{K}$ with $x\otimes y\in\cP(\cB)$ and $x\notin\cP(\cB)$ and let us show that $y\in\cP(\cB)$. Consider the $\otimes$-ideal~$\cat{C}=\SET{M\in\mK}{\hat{x}\otimes M\in\cat{B}}$. It is Serre by flatness of~$\hat{x}$ and the assumption $x\notin\cP(\cB)$ implies $\cat{B}\subseteq \cat{C}\neq\mK$. Therefore $\cat{C}=\cat{B}$ by maximality of~$\cat{B}$ and we get $y\in \yoneda\inv(\cat{C})=\yoneda\inv(\cat{B})=\cP(\cB)$.
\end{proof}

\begin{Rem}
\label{Rem:Spc'}%
We can push the analogy with the triangular spectrum $\SpcK$ a little further by considering the set $\Spch(\cat{K})$ of all homological primes:
\[
\Spch(\cK)=\SET{\cB\subset\mK}{\cB\textrm{ is a maximal Serre $\otimes$-ideal}}\,.
\]
We call it the \emph{homological spectrum of~$\cat{K}$} and equip it with a topology having as basis of closed subsets the following subsets~$\supph(x)$, one for every $x\in\cat{K}$:
\[
\supph(x):=\SET{\cB\in\Spch(\cat{K})}{\hat{x}\notin\cB}=\SET{\cB\in\Spch(\cat{K})}{\bar x\neq0\textrm{ in }\AKB}.
\]
One can verify that this pair $(\Spch(\cat{K}),\supph)$ is a \emph{support data} on~$\cat{K}$, in the sense of~\cite{Balmer05a}. Hence there exists a unique continuous map
\[
\phi\colon \Spch(\cat{K})\to \SpcK
\]
such that $\supph(x)=\phi\inv(\supp(x))$ for every~$x\in \cat{K}$. The explicit formula for~$\phi$ (\cite[Thm.\,3.2]{Balmer05a}) shows that $\phi$ is exactly the map $\cB\mapsto \yoneda\inv(\cB)$ of \Cref{Prop:field-to-prime}. We prove in \Cref{Cor:surj} below that this comparison map is surjective, at least for~$\cat{K}$ rigid. In fact, there are many known examples where $\phi$ is bijective. See \Cref{sec:examples}.
\end{Rem}

\begin{Def}
\label{Def:field-above-W}%
It will be convenient to say that a homological prime $\cB\in\Spch(\cK)$ \emph{lives over} a given subset $W\subseteq\SpcK$ of the triangular spectrum if the prime $\cP(\cB)=\yoneda\inv(\cB)$ of \Cref{Prop:field-to-prime} belongs to~$W$. By extension, we shall also say in that case that the corresponding homological residue field $\boneda=\boneda_{\cB}$ \emph{lives over~$W$}.
\end{Def}

In order to show surjectivity of~$\phi$, we derive from \Cref{Prop:Mod-quotients} the following:

\begin{Cor}
\label{Cor:Spc'(K/J)}%
Let $\cJ\subset\cK$ be a thick $\otimes$-ideal. With notation as in \Cref{Prop:Mod-quotients}, there is an inclusion-preserving one-to-one correspondence $\cat{C}\mapsto (q_!)\inv(\cat{C})$ between (maximal) Serre $\otimes$-ideals $\cat{C}$ of~$\mmod\cat{L}$ and the (maximal) Serre $\otimes$-ideals $\cB$ of~$\mK$ which contain~$\yoneda(\cJ)$; the inverse is given by $\cB\mapsto q_!(\cB)$. If $\cat{C}$ corresponds to~$\cB$ then the residue categories are canonically equivalent $\bat{A}(\cat{K};\cat{B})\cong \bat{A}(\cat{L};\cat{C})$ in such a way that the following diagram commutes up to canonical isomorphism
\[
\xymatrix@R=1.5em{
\cat{K} \ \ar@{^(->}[r]^-{\yoneda} \ar@/^1.5em/@<.4em>[rrr]^-{\boneda_{\cB}} \ar[d]_-{q}
& \mK \ar@{->>}[d]_-{q_!}\ \ar@{^(->}[r]^-{}
& \MK \ar@{->>}[r]^-{Q_\cB} \ar@{->>}[d]_-{q_!}
& \bat{A}(\cat{K};\cat{B}) \ar@{=}[d]^-{\cong}
\\
\cat{L} \ \ar@{^(->}[r]^-{\yoneda} \ar@/_1.5em/[rrr]_-{\boneda_{\cat{C}}}
& \mmod\cat{L} \ \ar@{^(->}[r]^-{}
& \MMod\cat{L} \ar@{->>}[r]^-{Q_{\cat{C}}}
& \bat{A}(\cat{L};\cat{C})\,.
}
\]
\end{Cor}

\begin{proof}
Standard `third isomorphism theorem' about ideals in a quotient.
\end{proof}

\begin{Rem}
In the notation of \Cref{Rem:Spc'}, the above \Cref{Cor:Spc'(K/J)} can be re\-phrased as saying that, for $\cL=\cK/\cJ$ or $(\cat{K}/\cJ)^\natural$, the map $\Spch(q_!)\colon\cat{C}\mapsto (q_!)\inv\cat{C}$ yields a homeomorphism between $\Spch(\cL)$ and the subspace $\SET{\cB\in\Spch(\cat{K})}{\cP(\cB)\supseteq\cJ}$ of $\Spch(\cK)$. In other words, the following commutative diagram is cartesian\,:
\[
\xymatrix@C=4em@R=1.5em{
\Spch(\cL) \ \ar@{^(->}[r]^-{\Spch(q_!)} \ar[d]_-{\phi}
& \Spch(\cK) \ar[d]^-{\phi}
\\
\Spc(\cL) \ \ar@{^(->}[r]^-{\Spc(q)}
& \SpcK\,.
}
\]
In the terminology of \Cref{Def:field-above-W}, the homological primes (or the residue fields) of~$\cL=\cK/\cJ$ or $\cL=(\cK/\cJ)^\natural$ canonically correspond to those of~$\cK$ which live `above' the subset $W(\cJ)=\SET{\cP\in\SpcK}{\cJ\subseteq \cP}\cong\Spc(\cat{L})$ of~$\SpcK$.
\end{Rem}

Let us prove an analogue of an old tensor-triangular friend~\cite[Lem.\,2.2]{Balmer05a}:
\begin{Lem}
\label{Lem:exist}%
Suppose that $\cK$ is rigid. Let $\cat{S}\subset \cK$ be a $\otimes$-multiplicative class of objects (\ie $\unit\in \cat{S}\supseteq \cat{S}\otimes \cat{S}$) and let $\cB_0\subset\cA\fp=\mK$ be a Serre $\otimes$-ideal which avoids~$\cat{S}$, that is, $\cB_0\cap \yoneda(\cat{S})=\varnothing$. Then there exists $\cB\subset\cA\fp$ a maximal Serre $\otimes$-ideal such that $\cB_0\subseteq\cB$ and $\cB$ still avoids~$\cat{S}$.
\end{Lem}

\begin{proof}
By Zorn, there exists $\cB$ maximal among the Serre $\otimes$-ideals which avoid~$\cat{S}$ and contain~$\cat{B}_0$. So we have $\cB\supseteq\cB_0$ and $\cB\cap \yoneda(\cat{S})=\varnothing$ and we are left to prove that $\cB$ is plain maximal in~$\cA\fp$. Consider $\cat{B}':=\SET{M\in\cat{A}\fp}{\hat x\otimes M\in\cB\textrm{ for some }x\in \cat{S}}$. Since $\cat{S}$ is $\otimes$-multiplicative and each $\hat x$ is $\otimes$-flat, the subcategory $\cB'\subsetneq\cA\fp$ is a Serre $\otimes$-ideal avoiding~$\cat{S}$ and containing~$\cB_0$. By maximality of~$\cB$ among those, the relation $\cat{B}\subseteq\cat{B}'$ forces $\cat{B}=\cB'$. In particular, $M=\ker(\hat\eta_x\colon \hat \unit \to \hat x^\vee\otimes \hat x)$ belongs to~$\cB$ for every $x\in \cat{S}$ since $\hat{x}\otimes M=0$ by rigidity (\Cref{Rem:rigid}). Let us show that $\cB\subset\cat{A}\fp$ is a maximal Serre $\otimes$-ideal by showing that a strictly bigger Serre $\otimes$-ideal $\cat{C}\supsetneq\cat{B}$ of~$\cat{A}\fp$ must be~$\cA\fp$ itself. Since $\cB$ is maximal among those avoiding~$\cat{S}$ and containing~$\cB_0$, such a strictly bigger~$\cat{C}$ cannot avoid~$\cat{S}$. Therefore $\cat{C}$ contains some $\hat x$ for $x\in \cat{S}$ and, by the above discussion, we also have $\ker(\hat \eta_x\colon \hat \unit\to \hat x^\vee\otimes \hat x)\in\cB\subseteq\cat{C}$. So in the exact sequence $0\to \ker(\hat \eta_x) \to \hat\unit \to \hat x^\vee\otimes\hat x$ we have $\ker(\hat\eta_x)$ and $\hat x$ in~$\cat{C}$. This forces $\hat \unit\in \cat{C}$ by Serritude and therefore $\cat{C}=\cat{A}\fp$ as wanted.
\end{proof}

\begin{Cor}
\label{Cor:surj}%
Suppose that $\cat{K}$ is rigid. Then the map $\cB\mapsto \cP(\cB)$ from homological primes to triangular primes as in \Cref{Prop:field-to-prime} (\ie the comparison map $\phi\colon\Spch(\cat{K})\to \SpcK$ of \Cref{Rem:Spc'}) is surjective. That is, every triangular prime~$\cat{P}\in\SpcK$ is of the form $\cat{P}=\yoneda\inv(\cB)$ for some maximal Serre $\otimes$-ideal~$\cB$ in~$\mK$.
\end{Cor}

\begin{proof}
Consider the quotient $\cat{K}/\cP$ (or its idempotent completion $\cat{K}_\cP:=(\cK/\cP)^\natural$). The map $\Spc(q)\colon \Spc(\cat{K}/\cat{P})\hook \SpcK$ sends $0$ to $q\inv(0)=\cP$. So, by \Cref{Cor:Spc'(K/J)} applied to $\cat{J}=\cat{P}$, it suffices to prove the result for $\cP=0$. In that case, $\cat{K}$ is \emph{local}, meaning that $0$ is a prime: $x\otimes y=0\Rightarrow x=0$ or $y=0$. We conclude by Lemma~\ref{Lem:exist} for $\cB_0=0$ and $\cat{S}=\cat{K}\oursetminus\{0\}$ which is $\otimes$-multiplicative because $\cK$ is local. (\footnote{\,The argument was already used in~\cite[Cor.\,4.9]{BalmerKrauseStevenson19}, which in turn inspired \Cref{Lem:exist}.})
\end{proof}

\begin{Rem}
\label{Rem:fields-vs-fields}%
There are a few differences between our approach to homological residue fields and the treatment in~\cite[\S\,4]{BalmerKrauseStevenson19}. First, the whole~\cite{BalmerKrauseStevenson19} is written for a `big' (\ie rigidly-compactly generated) tensor-triangulated category~$\cat{T}$ and the modules are taken over its rigid-compact objects~$\cat{K}:=\cat{T}^c$. This restriction is unimportant, certainly as far as most examples are concerned.

Another difference is that, in \loccit, we focussed on a \emph{local} category in the sense that $\SpcK$ is a local space, \ie has a unique closed point~$\cat{M}=0$. We then considered quotients of the module category $\cat{A}\onto \cat{A}/\Locat{B}$ for $\cat{B}\subseteq\cat{A}\fp$ maximal \emph{among those which meet~$\cat{K}$ trivially}, \ie $\cat{B}\cap \yoneda(\cat{K})=\{0\}$. This property means that the homological prime $\cB$ lives \emph{above the closed point} of~$\SpcK$ in the sense of \Cref{Def:field-above-W}. Equivalently, it means that the functor $\boneda_{\cB}\colon\cK\to \AKB$ is conservative, \ie detects isomorphisms. All these properties are reminiscent of commutative algebra, where the residue field of a local ring~$R$ is indeed conservative on perfect complexes and maps the unique prime of the field to the closed point of~$\Spec(R)$.

Continuing the analogy with commutative algebra, when dealing with a \emph{global} (\ie not necessarily local) category~$\cat{K}$, we can analyze it one prime at a time. For each $\cP\in\SpcK$ we can consider the local category~$\cat{K}_\cP=(\cat{K}/\cP)^\natural$. By \Cref{Cor:Spc'(K/J)} we can identify the homological primes~$\cat{C}$ of this local category~$\cK_\cP$ with a subset of those of the global category. Requesting that the local prime~$\cat{C}$ lives `above the closed point' of~$\Spc(\cK_\cP)$ as we did in~\cite{BalmerKrauseStevenson19} amounts to requesting that the corresponding global prime $\cB=(q_!)\inv(\cat{C})$ lives exactly above the point~$\cP$ in~$\SpcK$.

In other words, in \Cref{Def:homol-fields} we are considering all homological residue fields $\boneda_{\cB}\colon\cK\to \AKB$ at once but we can also regroup them according to the associated triangular primes $\yoneda\inv(\cB)$, in which case we obtain the constructions of~\cite{BalmerKrauseStevenson19} for the local category~$\cat{K}_\cP$.
\end{Rem}

%------------------------------------------------------------------------------
\bigbreak\goodbreak
\section{The Nilpotence Theorems}
\label{sec:Nil-Thms}%
\medbreak
%------------------------------------------------------------------------------

In this section, we assume $\cat{K}$ rigid.

Let us prove \Cref{Thm:nil-flat-intro}, in a strong form `with parameter'. Recall that an object $F$ in a tensor abelian category~$\cA$ is \emph{$\otimes$-flat} if $F\otimes-\colon\cA\to \cA$ is exact.

\begin{Thm}
\label{Thm:main}%
Let $\cat{K}$ be an essentially small, rigid tensor-triangulated category (\Cref{Hyp:0}) and $W\subseteq\SpcK$ a closed subset (the `parameter'). Let $f\colon \hat{x}\to F$ be a morphism in $\cA=\MK$ satisfying the following hypotheses:
\begin{enumerate}[\rm(i)]
\item
The source of~$f$ comes from an object $x\in\cat{K}$ via Yoneda, as indicated above.
\item
Its target $F$ is $\otimes$-flat in~$\cA$.
\item
\label{it:nil-iii}%
The morphism $f$ \emph{vanishes in every homological residue field over~$W$} in the following sense: For every maximal Serre $\otimes$-ideal $\cB\subset\mK$ living over~$W$ (\Cref{Def:field-above-W}) we have $Q_\cB(f)=0$ in $\AKB=(\MK)/\Locat{B}$.
\end{enumerate}
Then:
\begin{enumerate}[\rm(a)]
\item
\label{it:nil-a}%
There exist an object $s\in\cat{K}$ such that $\supp(s)\supseteq W$ and an integer $n\ge 1$ such that $f\potimes{n}\otimes \hat{s}=0$ in~$\cat{A}$.
\smallbreak
\item
\label{it:nil-b}%
For any object $s$ as above, if we let $Z=\supp(s)\supseteq W$, then for every $z\in\cK_Z=\SET{z\in\cK}{\supp(z)\subseteq Z}$ there exists $m\ge 1$ with $f\potimes{m}\otimes\hat{z}=0$.
\end{enumerate}
\end{Thm}

\begin{proof}
By \Cref{Rem:adj-potimes}, we can and shall assume that $x=\unit$. So $f\colon\hat\unit\to F$. Consider
\[
\cat{S}:=\SET{s\in \cK}{\supp(s)\supseteq W}\,.
\]
This is a $\otimes$-multiplicative class of objects of~$\cK$ since $\supp(s_1\otimes s_2)=\supp(s_1)\cap \supp(s_2)$. Since $W$ is closed and $\{\supp(s)\}_{s\in\cK}$ is a basis of closed subsets, we have $\cap_{s\in \cat{S}}\supp(s)=W$. On the other hand, consider the following subcategory of finitely presented $\cK$-modules
\[
\cat{B}_0:=\SET{M\in\cat{A}\fp}{f\potimes{n}\otimes M=0\textrm{ in }\cat{A}\textrm{ for some }n\ge 1}.
\]
Note that $\cat{B}_0$ is a Serre $\otimes$-ideal. This uses that $F$ is $\otimes$-flat in~$\cat{A}$ and was already proved in~\cite[Lemma~4.17]{BalmerKrauseStevenson19}. In particular, when we prove that $\cB_0$ is closed under extension, if $0\to M_1\to M_2\to M_3\to 0$ is an exact sequence in~$\cA\fp$ and if $f\potimes{n_1}\otimes M_1=0$ and $f\potimes{n_3}\otimes M_3=0$ then we show that $f\potimes{(n_1+n_3)}\otimes M_2=0$. This is the place where nilpotence is needed, as opposed to mere vanishing.

If $\cB_0$ meets $\yoneda(\cat{S})$ we obtain the conclusion of part~\eqref{it:nil-a}. Suppose \ababs, that $\cB_0\cap \yoneda(\cat{S})=\varnothing$. By \Cref{Lem:exist} there exists a homological prime $\cB\in\Spch(\cK)$ containing $\cB_0$ and still avoiding~$\cat{S}$. The latter property $\cB\cap \yoneda(\cat{S})=\varnothing$ means that the triangular prime $\yoneda\inv(\cB)$ belongs to the subset $\SET{\cP\in\SpcK}{\cP\cap \cat{S}=\varnothing}=\SET{\cP\in\SpcK}{s\notin\cP, \ \forall\ s\in \cat{S}}=\cap_{s\in \cat{S}}\supp(s)=W$, as we proved above from the definition of~$\cat{S}$. So $\cB$ lives over~$W$ and we trigger hypothesis~\eqref{it:nil-iii} for that~$\cB$, namely that we have $Q_\cB(f)=0$ in~$\AKB$.

Consider now the kernel of~$f\colon \hat\unit\to F$ in~$\cA$ and the commutative diagram:
\[
\xymatrix{
0 \ar[r]
& \ker(f) \ar[r]^-{j} \ar[d]_-{f\otimes 1}
& \hat \unit \ar[r]^-{f} \ar[d]_-{f}
& F \ar[d]_-{f\otimes 1}
\\
0 \ar[r]
& F\otimes \ker(f) \ar[r]^-{1\otimes j}
& F \ar[r]^-{1\otimes f}
& F\otimes F
\\
}
\]
whose first row is the obvious one; the diagram is obtained by tensoring that first row with $f\colon \hat \unit\to F$ itself (on the left) and using that $\hat\unit$ is the $\otimes$-unit in~$\cat{A}$. It is essential here that the source of~$f$ is~$\hat\unit$ (which we achieved through rigidity), and not some random object. Indeed, the diagonal of the left-hand square is now $f\circ j=0$. Since the lower row is exact ($F$ $\otimes$-flat again), we conclude that $f\otimes \ker(f)=0$. We cannot jump to the conclusion that $\ker(f)\in\cB$ since $\ker(f)$ is not finitely presented. However, $\ker(f)\into \hat \unit$ is a sub-object of a finitely presented object, hence it is the \emph{union} of its finitely presented subobjects as in~\cite[Lemma~3.9]{BalmerKrauseStevenson17app}, \ie
\[
\ker(f)=\colim_{\genfrac{}{}{0pt}{1}{M\into \ker(f)}{M\in\cat{A}\fp}} M\,.
\]
For any such $i\colon M\into \ker(f)$ with $M\in\cat{A}\fp$, we have a commutative square obtained by tensoring $i\colon M\into \ker(f)$ with $f\colon \hat \unit\to F$:
\[
\xymatrix{
M \ar@{ >->}[r]^-{i} \ar[d]_-{f\otimes 1}
& \ker(f) \ar[d]^-{f\otimes 1=0}
\\
F\otimes M \ar@{ >->}[r]^-{1\otimes i}
& F\otimes \ker(f)\,.
}
\]
Note that the bottom map remains a monomorphism because $F$ is $\otimes$-flat. The vanishing of the right-hand vertical map, proved above, gives us $f\otimes M=0$, which means $M\in\cat{B}_0\subseteq\cat{B}$. It follows that $\ker(f)$ is a colimit of objects~$M\in\cB$ and therefore belongs to~$\Locat{B}$. Applying the exact functor $Q_{\cat{B}}\colon \cat{A}\onto \bat{A}=\cat{A}/\Locat{B}$ to the morphism $f$, we have just proved that $Q_{\cat{B}}(f)$ has trivial kernel, \ie it is a monomorphism $Q_{\cat{B}}(\hat \unit)\into Q_{\cat{B}}(F)$ in~$\bat{A}$. But this monomorphism $Q_{\cat{B}}(f)$ is also zero by assumption~\eqref{it:nil-iii} on~$f$ that we triggered in the first part of the proof. This forces $Q_{\cB}(\hat\unit)=0$ and thus $\hat\unit\in\cB$, a contradiction.

This finishes the proof of part~\eqref{it:nil-a}. To deduce~\eqref{it:nil-b}, we use the fact that $\cB_0=\SET{M\in\cat{A}\fp}{f\potimes{m}\otimes M=0\textrm{ for some }m\ge1}$ is a Serre $\otimes$-ideal, as we saw above, by~\cite[Lem.\,4.17]{BalmerKrauseStevenson19}. Therefore $\yoneda\inv(\cat{B}_0)$ is a thick $\otimes$-ideal of~$\cK$ and so the fact that $\yoneda\inv(\cB_0)$ contains~$s$ implies that it contains the whole thick $\otimes$-ideal of~$\cK$ generated by~$s$, namely exactly $\cat{K}_{Z}$ where $Z=\supp(s)$, see~\cite[\S\,4]{Balmer05a}.
\end{proof}

We can then deduce the form announced in \Cref{Thm:nil-flat-intro}:

\begin{Cor}
\label{Cor:main}%
Let $f\colon \hat x\to F$ be a morphism in $\cA=\MK$, with $x\in\cat{K}$ and with $F$ $\otimes$-flat in~$\cA$. Suppose that for every homological prime $\cB\subset\mK$ we have $Q_\cB(f)=0$ in $\AKB$. Then there exists $m\ge 1$ such that $f\potimes{m}=0$ in~$\cat{A}$.
\end{Cor}

\begin{proof}
Apply \Cref{Thm:main} for $W=\SpcK$. Thus the subset $Z\supseteq W$ of Part~\eqref{it:nil-b} must be $Z=\SpcK$, and we can take $z=\unit$.
\end{proof}

\begin{Rem}
We stress that the `parameter'~$W$ in \Cref{Thm:main} is more flexible than the `parameter' in~\cite[Thm.\,3.8]{Thomason97}, where the closed subset $W$ is supposed to be the support of some object, \ie a Thomason closed subset (\Cref{Rem:tt}). In that case, any object $s$ with $\supp(s)=W$ will do, as follows from~\Cref{Thm:main}\,\eqref{it:nil-b}.
\end{Rem}

Of course, the easiest source of $\otimes$-flat objects in~$\MK$ is the Yoneda embedding:
\begin{Cor}
\label{Cor:main-K}%
Let $f\colon x\to y$ be a morphism in~$\cK$ and $W\subseteq\SpcK$ be a closed subset such that $\boneda_{\cB}(f)=0$ for every homological residue field corresponding to a homological prime~$\cB$ living above~$W$ (\Cref{Def:field-above-W}). Then there exists a Thomason closed subset $Z\supseteq W$ (\Cref{Rem:tt}) with the property that for every $z\in\cat{K}$ such that $\supp(z)\subseteq Z$ there exists $n\ge 1$ with $f\potimes{n}\otimes z=0$ in~$\cat{K}$. In particular, this holds for \emph{some} $z$ with $\supp(z)=Z\supseteq W$.
\end{Cor}

\begin{proof}
This is \Cref{Thm:main} for $F=\hat y$, which is $\otimes$-flat, combined with faithfulness of Yoneda $\yoneda\colon \cat{K}\hook\cat{A}$ to bring the conclusion back into~$\cK$.
\end{proof}

This specializes to the flagship Nilpotence \Cref{Thm:nil-compact}:
\begin{Cor}
\label{Cor:flagship}%
Let $f\colon x\to y$ be a morphism in~$\cK$ such that $\boneda(f)=0$ for every homological residue field~$\boneda$ of~$\cK$. Then there exists $n\ge 1$ with $f\potimes{n}=0$ in~$\cK$.
\end{Cor}

\begin{proof}
\Cref{Cor:main-K} for $W=\SpcK$, hence $Z=\SpcK$, and $z=\unit$.
\end{proof}

\begin{Rem}
\label{Rem:big-T}%
Many of our examples of tensor-triangulated categories~$\cat{K}$, if not all, appear as the compact-rigid objects $\cat{K}=\cat{T}^c$ in some compactly-rigidly generated `big' tensor-triangulated category~$\cat{T}$. See~\cite[Hyp.\,0.1]{BalmerKrauseStevenson19}. In that case, we have a \emph{restricted-Yoneda} functor which extends $\yoneda\colon \cat{K}\hook \cat{A}=\MK$ to the whole of~$\cat{T}$:
\[
\xymatrix@C=4em@R=.5em{
\cK=\cat{T}^c \ \ar@{^(->}[r]^-{\yoneda} {}_\vcorrect{.6} \ar@{^(->}[dd]
& \cat{A}\fp=\mK {}_\vcorrect{.6} \ar@{^(->}[dd]
\\
\\
\cat{T}\ \ar@{->}[r]^-{\yoneda}
& \cat{A}=\MK
\\
X \ar@{|->}[r]
& \hat{X}:=\Homcat{T}(-,X)\,.\kern-3em
}
\]
Note that $\yoneda\colon \cat{T}\to \cat{A}$ is not faithful anymore (it kills the so-called \emph{phantom maps}). However, it is faithful for maps \emph{out of} a compact, by the usual Yoneda Lemma, that is, $\Homcat{T}(x,Y)\to \Homcat{A}(\hat x,\hat Y)$ is bijective as soon as $x\in\cat{K}$. We prove in~\cite[Prop.\,A.14]{BalmerKrauseStevenson17app} that every $\hat Y$ remains $\otimes$-flat in~$\MK$, even for $Y\in\cat{T}$ non-compact.

For every homological prime $\cB\in\Spch(\cK)$ we can still compose restricted-Yoneda $\cat{T}\to \cat{A}$ with $Q_\cB\colon \cA\onto \AKB$. We obtain a well-defined homological residue field on the whole `big' category~$\cat{T}$, that we still denote
\[
\boneda_{\cB}\colon \cat{T}\too \AKB\,.
\]
This remains a homological tensor-functor. Compare~\cite[Thm.\,1.6]{BalmerKrauseStevenson19}. Note that we may use these functors to define a support for big objects in~$\cat{T}$, as will be investigated elsewhere.
\end{Rem}

We can finally unpack our nilpotence theorems in that special case:
\begin{Cor}
\label{Cor:nil-big}%
Let $\cat{T}$ be a rigidly-compactly generated `big' tensor-triangulated category and $\cat{K}=\cat{T}^c$ as in \Cref{Rem:big-T}. Let $f\colon x\to Y$ be a morphism in~$\cT$ with $x\in\cK$ compact and $Y$ arbitrary.
\begin{enumerate}[\rm(a)]
\item
\label{it:nil-T-b}%
Suppose that $\boneda_{\cB}(f)=0$ in~$\AKB$ for every homological residue field~$\boneda_{\cB}$ of \Cref{Rem:big-T}, for every homological prime~$\cB\subset\mK$ of \Cref{Def:homol-fields}. Then we have $f\potimes{n}=0$ in~$\cT$ for some $n\ge 1$.
\smallbreak
\item
\label{it:nil-T-a}%
Suppose that $\boneda_{\cB}(f)=0$ in~$\AKB$ for every homological residue field over a closed subset~$W\subseteq\SpcK$ (\Cref{Def:field-above-W}). Then there exists a Thomason closed subset $Z\subseteq\SpcK$ such that $Z\supseteq W$ and such that for every $z\in\cat{K}_Z=\SET{z\in\cK}{\supp(z)\subseteq Z}$ we have $f\potimes{n}\otimes z=0$ for some $n\ge 1$. In particular, this holds for \emph{some} $z$ with $\supp(z)=Z\supseteq W$ (\Cref{Rem:tt}).
\end{enumerate}
\end{Cor}

\begin{proof}
This follows from \Cref{Thm:main} applied to $F=\hat Y$, together with the partial faithfulness of restricted-Yoneda explained in \Cref{Rem:big-T}.
\end{proof}

%------------------------------------------------------------------------------
\bigbreak\goodbreak
\section{Examples}
\label{sec:examples}%
\medbreak
%------------------------------------------------------------------------------

For this section, we keep the setting of \Cref{Rem:big-T}, that is, $\cT$ is a `big' tensor-triangulated category generated by its subcategory~$\cK=\cT^c$ of rigid-compact objects.
\begin{Rem}
\label{Rem:EB}%
We recall some of the tools developed in~\cite[\S\,3]{BalmerKrauseStevenson19}. Let $\cB$ be a proper Serre \tensid\ in~$\cA\fp=\mK$. Picking an injective hull of the unit $\bar\unit$ in the quotient~$\AKB$ yields a (pure-injective) object~$\EB$ in~$\cT$ such that
\begin{equation}
\label{eq:EB-<B>}%
\Locat{B}=\Ker(\hEB\otimes-)=\SET{M\in \MK}{\hEB\otimes M=0}\,.
\end{equation}
In fact the object~$\EB$ is a \emph{weak ring} in~$\cT$, \ie it comes with a map~$\eta_\cB\colon \unit\to \EB$ such that $\EB\otimes\eta_\cB\colon \EB\to \EB\otimes\EB$ is a split monomorphism. A retraction $\EB\otimes\EB\to \EB$ of this monomorphism can be viewed as a (non-associative) multiplication on~$\EB$ for which $\eta_\cB\colon \unit\to \EB$ is a right unit. In any case, one important property of $\eta_\cB$ is that it cannot be $\otimes$-nilpotent, for otherwise $\EB$ would be a retract of zero, hence zero, forcing~$\cB=\cA\fp$.
\end{Rem}

Another important feature of the objects~$\EB$, for~$\cB$ maximal, is the following:
\begin{Prop}
\label{Prop:EB@EB'=0}%
For distinct $\cB\neq\cB'$ in~$\SpchK$ we have $\EB\otimes\EBp=0$.
\end{Prop}

\begin{proof}
This is already in~\cite[Cor.\,4.12]{BalmerKrauseStevenson19}, at least in the local setting. Let us recall the idea, which is easy. In $\cA\fp=\mK$, the kernel $\Ker(\hEB\otimes\hEBp\otimes-)$ is a Serre \tensid\ containing both~$\cB$ and $\cB'$, which are maximal. If that kernel was a proper subcategory it would then be equal to both $\cB$ and~$\cB'$, thus forcing the forbidden~$\cB=\cB'$. So this kernel is not proper, \ie it contains the unit~$\hat\unit$. This reads $\hEB\otimes\hEBp=0$. Now restricted-Yoneda $\yoneda\colon\cT\to \MK$ is a conservative $\otimes$-functor, hence $\EB\otimes \EBp=0$ as claimed.
\end{proof}

Let us now prove the converse to \Cref{Cor:nil-big}.

\begin{Thm}
\label{Thm:converse2nil}%
Let $\cT$ be a `big' tensor-triangulated category with~$\cK=\cT^c$, as in \Cref{Rem:big-T}. Consider a family $\calF\subseteq\SpchK$ of points in the homological spectrum. Suppose that the corresponding functors $\hB\colon \cT\to \AKB$ collectively detect $\otimes$-nilpotence in the following sense: If $f\colon x\to Y$ in~$\cT$ is such that $x\in\cT^c$ and $\hB(f)=0$ for all~$\cB\in\calF$, then $f\potimes{n}=0$ for $n\gg1$. Then we have $\calF=\SpchK$.
\end{Thm}

\begin{proof}
Suppose that $\calF\neq\SpchK$. Then there exists $\cB'\in\SpchK$ which does not belong to~$\calF$. By \Cref{Prop:EB@EB'=0} we have $\EB\otimes\EBp=0$ for all~$\cB\in\calF$. Consider the map $\eta_{\cB'}\colon \unit\to \EBp$ as in \Cref{Rem:EB}. We have therefore proved that $\hB(\eta_{\cB'})=0$ for all~$\cB\in\calF$ for the obvious reason that its target object, $\EBp$, goes to zero along all~$\hB$. On the other hand, we have seen in \Cref{Rem:EB} that $\eta_{\cB'}$ cannot be $\otimes$-nilpotent. Hence a proper family~$\calF\subsetneq\SpchK$ cannot detect $\otimes$-nilpotence.
\end{proof}

\begin{Rem}
\label{Rem:converse2nil}%
The proof shows that it is enough to assume the property that the family $\{\hB\}_{\cB\in\calF}$ detects the vanishing of objects~$Y\in \cT$. Compare \Cref{Rem:oddball}.
\end{Rem}

Here is the picture we will observe in several examples:

\begin{Thm}
\label{Thm:phinjective}%
Let $\cT$ be a `big' tensor-triangulated category and $\cK=\cT^c$ as in \Cref{Rem:big-T}. Suppose given for every point $\cP\in\SpcK$ of the triangular spectrum, a coproduct-preserving, homological and (strong) monoidal functor
\[
H_\cP\colon \cT\to \cA_\cP
\]
with values in a tensor-abelian category~$\cA_\cP$ and satisfying the following properties\,:
\begin{enumerate}[\rm(1)]
\item
\label{it:AP}%
For each $\cP$, the target~$\cA_\cP$ is a locally coherent (\Cref{Rem:Freyd}) Grothendieck category with colimit-preserving tensor; the subcategory $\cA_\cP\fp$ of finitely presented objects is \emph{simple} in the sense that its only Serre \tensid{}s are~0 and~$\cA_\cP\fp\neq0$.
\smallbreak
\item
\label{it:HP}%
The functor $H_\cP\colon \cT\to \cA_\cP$ maps compacts to finitely presented: $H_\cP(\cK)\subseteq\cA_\cP\fp$. Furthermore, it maps every $X\in\cT$ to a $\otimes$-flat object in~$\cA_\cP$. Finally the thick $\otimes$-ideal $\Ker(H_\cP)\cap \cK=\SET{x\in\cK}{H_\cP(x)=0}$ is equal to~$\cP$.
\smallbreak
\item
\label{it:detect}%
The family $\{H_\cP\}_{\cP\in\SpcK}$ detects $\otimes$-nilpotence of maps $f\colon x\to Y$ in~$\cT$, with $x\in\cK$ compact: If $H_\cP(f)=0\textrm{ for all }\cP$ then $f\potimes{n}=0\textrm{ for }n\gg1$.
\end{enumerate}
Then the comparison map~$\phi\colon \SpchK\to \SpcK$ of \Cref{Prop:field-to-prime} is a bijection.
\end{Thm}

\begin{proof}
The core of the proof amounts to the kernel of each~$H_\cP$ defining an element of~$\SpchK$. More precisely, let us fix~$\cP\in\SpcK$ for the moment and denote by
\[
\hat H_\cP\colon \cA\too \cA_\cP
\]
the unique coproduct-preserving \emph{exact} functor that extends~$H_\cP$ to~$\cA=\MK$, that is, such that the following diagram commutes\,:
\begin{equation}
\label{eq:hatH_P}%
\vcenter{\xymatrix{
\cT\ar[rd]_-{H_\cP} \ar[r]^-{\yoneda}
& \cA=\MK \ar[d]^-{\hat H_\cP}
\\
& \cA_\cP\,.\!\!
}}
\end{equation}
The existence of such an~$\hat H_\cP$ was established by Krause~\cite[Cor.\,2.4]{Krause00}. It is not hard to show that $\hat H_\cP$ is also monoidal. At the very least, for every $M\in\cA$, we can find~$g\colon Y\to Z$ in~$\cT$ such that $M=\im(\hat g)$ and then exactness of~$\hat H_\cP$ gives
\begin{equation}
\label{eq:HP-tens}%
\hat H_\cP(\hat X\otimes M)\cong
H_\cP(X)\otimes \hat H_\cP(M)
\end{equation}
for every $X\in\cT$ and~$M\in\cA$. Consider now the kernel of~$\hat H_\cP$ in~$\cA$. The exact functor $\hat H_\cP\colon \cA\to \cA_\cP$ preserves coproducts and finitely presented objects because of~\eqref{it:HP}. It then follows by a general property of abelian categories that $\Ker(H_\cP)$ is generated by its finitely presented objects (see~\cite[Prop.\,A.6]{BalmerKrauseStevenson19}). In other words, if we define a Serre \tensid\ of~$\cA\fp$ as follows
\[
\cB(\cP):=\Ker(\hat H_\cP)\cap \cA\fp=\SET{M\in\mK}{\hat H_\cP(M)=0}
\]
then we have $\Ker(\hat H_\cP)=\Locat{B(\cP)}$. The fact that $\cA_\cP\neq0$ tells us that $\cB:=\cB(\cP)$ is proper. We claim that it is maximal. Let $\cB'$ be a strictly larger Serre \tensid
\[
\cB=\cB(\cP)\subsetneq\cB'\subseteq\cA\fp.
\]
We want to prove that $\cB'=\cA\fp$. Choose $M\in\cB'$ which does not belong to~$\cB$. Let us now invoke the objects~$\EB$ and $\EBp$ of~$\cT$, as in \Cref{Rem:EB}, so that
\[
\Locat{B}=\Ker(\hEB\otimes-)\qquad\textrm{and}\qquad\Locat{B'}=\Ker(\hEBp\otimes-)\,.
\]
We then have $\hEBp\otimes M=0$ because $M\in \cB'$. Hence by~\eqref{eq:HP-tens}, we have $H_\cP(\EBp)\otimes \hat H_\cP(M)=0$. On the other hand, we have $\hat H_\cP(M)\neq0$ because~$M\notin\cB$. Since $H_\cP(\EBp)$ is $\otimes$-flat in~$\cA_\cP$, we can consider the Serre \tensid\
\[
\Ker(H_\cP(\EBp)\otimes-)\cap \cA_\cP\fp
\]
of~$\cA_\cP\fp$. We just proved that it contains a non-zero object, namely $\hat H_\cP(M)$. By the `simplicity' of~$\cA_\cP\fp$, we get that $\Ker(H_\cP(\EBp)\otimes-)\cap\cA\fp$ must be the whole of~$\cA_\cP\fp$. This means that $H_\cP(\EBp)=0$, or in other words, $\hEBp\in\Ker(\hat H_\cP)=\Locat{\cB}=\Ker(\hEB\otimes-)$. We have thus proved
\begin{equation}
\label{eq:aux-EB@EB'}%
\hEB\otimes\hEBp=0\,.
\end{equation}
Consider now the exact sequence in~$\cA$ associated to the morphism~$\eta_\cB\colon \unit\to \EB$
\[
0\to I_{\cB} \to \hat \unit \xrightarrow{\hat\eta_{\cB}} \hEB\,.
\]
Since $\EB\otimes\eta_\cB$ is a split monomorphism, we have $I_{\cB}\in\Ker(\hEB\otimes-)=\Locat{B}\subseteq\Locat{B'}$ and therefore $I_\cB\otimes \hEBp=0$. Combined with~\eqref{eq:aux-EB@EB'} we see from the above exact sequence that $\hat \unit$ is also killed by~$\hEBp$, that is $\hEBp=0$, or $\cB'=\cA\fp$ as claimed.

In summary, we have now shown that $\cB(\cP)=\Ker(\hat H_\cP)\cap \cA\fp$ belongs to the homogeneous spectrum~$\SpchK$. By the last assumption in~\eqref{it:HP}, we see that $\phi(\cB(\cP))=\cP$. Finally, we need to relate the functor~$H_\cP$ with the homological residue field~$\hB$ for $\cB:=\cB(\cP)$. This is now easy. From~\eqref{eq:hatH_P} and the fact that $\Locat{B}=\Ker(\hat H_\cP)$ we can further factor~$\hat H_\cP$ by first modding out this kernel. We obtain a unique exact functor $\bar H_\cP\colon \AKB\to \cA_\cP$ making the right-hand triangle in the following diagram commute\,:
\[
\vcenter{\xymatrix{
\cT\ar[rd]_-{H_\cP} \ar[r]^(.6){\yoneda} \ar@/^2em/[rr]^-{\hB}
& \cA \ar[d]^-{\hat H_\cP} \ar@{->>}[r]^-{Q_\cB}
& \AKB=\mK/\Locat{B} \ar[ld]^-{\bar H_\cP}
\\
& \cA_\cP\,.\!\!
}}
\]
The left-hand triangle was already in~\eqref{eq:hatH_P}. The top `triangle' commutes by definition of~$\hB\colon \cT\to \AKB$. Expanding the notation $\cB=\cB(\cP)$, we have factored each $H_\cP\colon \cT\to \cA_\cP$ via a homological residue field~$\boneda_{\cB(\cP)}$ as follows
\[
H_\cP=\bar H_{\cP}\circ\boneda_{\cB(\cP)}.
\]

We now claim that the family
\[
\calF:=\SET{\cB(\cP)}{\cP\in\SpcK}
\]
satisfies the hypothesis of \Cref{Thm:converse2nil}, in other words that the family of functors $\{\boneda_{\cB(\cP)}\}_{\cP\in\SpcK}$ detects $\otimes$-nilpotence of maps $f\colon x\to Y$ in~$\cat{T}$ with $x\in\cK$ compact. Indeed if $\boneda_{\cB(\cP)}(f)=0$ then $H_\cP(f)=\bar H_{\cP}\circ \boneda_{\cB(\cP)}(f)=0$ by the above factorization. If this holds for all~$\cP$, we conclude by~\eqref{it:detect} that $f\potimes{n}=0$ for~$n\gg1$. So \Cref{Thm:converse2nil} tells us that this family $\calF=\SET{\cB(\cP)}{\cP\in\SpcK}$ is the whole~$\SpchK$.

In conclusion, we have constructed a set-theoretic section
\[
\sigma\colon \SpcK\to \SpchK, \qquad \cP\mapsto \cB(\cP)
\]
of $\phi\colon \SpchK\onto \SpcK$ and we just proved that $\im(\sigma)=\calF=\SpchK$. In other words, the surjection $\phi$ admits a surjective section, \ie~$\phi$ is a bijection.
\end{proof}

We can now use known nilpotence-detecting families in examples, to prove that $\phi\colon \SpchK\to \SpcK$ is bijective.

\begin{Cor}
\label{Cor:SH}%
Let $\cT=\SH$ be the stable homotopy category and $\cK=\SH^c$. Then $\phi\colon \SpchK\to \SpcK$ is a bijection. More generally, let $G$ be a compact Lie group and $\cT=\SH(G)$ the $G$-equivariant stable homotopy category of genuine $G$-spectra, and $\cK=\SH(G)^c$. Then $\phi\colon \SpchK\to \SpcK$ is a bijection.
\end{Cor}

\begin{proof}
In the case of~$\cT=\SH$, this relies on~\cite{DevinatzHopkinsSmith88,HopkinsSmith98}. As explained in~\cite[\S\,9]{Balmer10b}, the spectrum consists of points~$\cP(p,n)$ for each prime number~$p$ and each `chromatic height' $1\le n\le\infty$, with the collision $\cP(0):=\cP(p,1)=\SH^{\mathrm{tor}}$ for all~$p$. This prime~$\cP(0)$ is the kernel of rational homology $H\bbQ\otimes-\colon \SH^c\to \Db(\bbQ)\cong\bbQ\GGrmod$. The other primes~$\cP(p,n)$ for $n\ge 2$ are given as the kernels of Morava $K$-theory~$K(p,n-1)_\sbull$, which are homological functors
\[
K(p,n)_\sbull\colon \SH\to \cA_{p,n}:=\bbF_p[v_n^{\pm1}]\GGrmod
\]
for $1\le n<\infty$ with $v_n$ of degree~$2(p^n-1)$, and $K(p,\infty)_\sbull\colon \SH\to \bbF_p\GGrmod$ is mod-$p$ homology. The target categories are graded modules over (graded) fields and the Morava $K$-theories satisfy K\"unneth formulas, which amounts to say that they are monoidal functors when $\cA_{p,n}$ is equipped with the graded tensor product. See~\cite{Ravenel92}. The reader can now verify Conditions~\eqref{it:AP}--\eqref{it:detect} of \Cref{Thm:phinjective}. The crucial~\eqref{it:detect} is the original Nilpotence Theorem~\cite{DevinatzHopkinsSmith88}.

For $\cT=\SH(G)$ and $\cK=\SH(G)^c$, the description of the set~$\SpcK$ and the relevant nilpotence theorem was achieved for finite groups in~\cite{BalmerSanders17}, and more recently for arbitrary compact Lie groups in~\cite{BarthelGreenleesHausmann18pp}. Specifically, there is exactly one prime~$\cP(H,p,n)=(\Phi^H)\inv(\cP(p,n))$ in~$\SpcK$ for every conjugacy class of closed subgroups~$H\le G$ and for every `chromatic' prime~$\cP(p,n)$ as above; here $\Phi^H\colon \SH(G)\to \SH$ is the geometric $H$-fixed point functor, which is tensor-triangulated. The relevant homology theories are simply these $\Phi^H$ composed with the non-equivariant Morava $K$-theories. So Conditions~\eqref{it:AP}--\eqref{it:HP} in \Cref{Thm:phinjective} are easy to verify. The relevant nilpotence theorem giving us~\eqref{it:detect} can be found in~\cite[Thm.\,3.12]{BarthelGreenleesHausmann18pp} (or \cite[Thm.\,4.15]{BalmerSanders17} for finite groups, a result also obtained earlier by N.\ Strickland).
\end{proof}

\begin{Cor}
\label{Cor:Dperf}%
Let $X$ be a quasi-compact and quasi-separated scheme and $\cT=\Der(X)$ the derived category of $\calO_X$-modules with quasi-coherent homology. Here $\cK=\Dperf(X)$ is the category of perfect complexes, the spectrum $\SpcK\cong |X|$ is the underlying space of~$X$ and the map~$\phi\colon\SpchK\to \SpcK$ is a bijection.
\end{Cor}

\begin{proof}
The homological functors of \Cref{Thm:phinjective} are simply the classical residue fields $\kappa(x)\otimes_{\calO_X}^{\mathrm{L}}-\colon \Der(X)\to \Der(\kappa(x))\cong\kappa(x)\GGrmod$ at each $x\in |X|$, where of course $\kappa(x)$ is the residue fields of the local ring~$\calO_{X,x}$. Here, the relevant nilpotence theorem is due to Thomason~\cite[Thm.\,3.6]{Thomason97}.
\end{proof}

\begin{Rem}
\label{Rem:oddball}%
There is a simpler proof of the above when $X$ is noetherian, following the pattern of the next example. It is worth noting that when $X$ is not noetherian, even for $|X|=\ast$, the residue fields do not detect vanishing of objects. See~\cite{Neeman00}.
\end{Rem}

\begin{Exa}
\label{Exa:Stab}%
Let $G$ be a finite group scheme over a field~$k$ and $\cT=\Stab(kG)$ the category of $k$-linear representations of~$G$ modulo projectives. See~\cite{BensonIyengarKrausePevtsova18}. Here $\cK=\stab(kG)$ is the stable category of finite-dimensional representations modulo projectives, $\SpcK\cong\Proj(\mathrm{H}^\sbull(G,k))$ is the so-called projective support variety, and the map $\phi\colon \SpchK\to \SpcK$ is again bijective. The method of proof is different, for there is no known homology theories capturing the points of~$\SpcK$ \emph{and satisfying a K\"unneth formula}. Indeed, points are detected by equivalence classes of so-called \emph{$\pi$-points}, following~\cite{FriedlanderPevtsova07}, but these functors are \emph{not} monoidal!

Instead, we can use the fact that localizing subcategories of~$\cT$ are classified by subsets of~$\SpcK$ in this case, a non-trivial result that can be found in~\cite[\S\,10]{BensonIyengarKrausePevtsova18}. In such situations, the property $\EB\otimes\EBp=0$ isolated in \Cref{Prop:EB@EB'=0}, for $\cB\neq\cB'$ in~$\SpchK$ can be used to show that $\phi\colon \SpchK\to \SpcK$ is injective. Indeed, if~$\phi(\cB)=\phi(\cB')=:\cP$, we can use minimality of the localizing category~$\cT_{\cP}$ supported at the point~$\cP$ to show that $\EB\otimes\EBp=0$ forces $\EB=0$ or~$\EBp=0$ which is absurd. This argument can already be found in~\cite[Cor.\,4.26]{BalmerKrauseStevenson19}.
\end{Exa}

\begin{Rem}
\label{Rem:stab-nil}%
We note that the above does \emph{not} use a nilpotence theorem for $\cT=\Stab(kG)$. To the best of the author's knowledge, there is no such result in the literature, the probable reason being that $\pi$-points (or shifted cyclic subgroups) are not monoidal. Thanks to the present work, we now know that there exists for every point $\cP\in\Spc(\stab(kG))\cong\Proj(\mathrm{H}^\sbull(G,k))$ a unique homological \emph{tensor} functor
\[
\boneda_{\cB(\cP)}\colon \Stab(kG)\to \bat{A}(\cat{K};\cat{B}(\cP))
\]
to a `simple' Grothendieck tensor category (in the sense of~\eqref{it:AP} in~\Cref{Thm:phinjective}), whose kernel on compacts~$\cK=\cT^c=\stab(kG)$ is exactly~$\cP$. And we know that the family $\{\boneda_{\cB(\cP)}\}_{\cP\in\SpcK}$ detects tensor-nilpotence.
\end{Rem}

\begin{Rem}
In view of the avalanche of examples where $\phi\colon \SpchK\to \SpcK$ is a bijection, it takes nerves of steel not to conjecture that this property holds for all tensor-triangulated categories.
\end{Rem}

%%------------------------------------------------------------------------------
%\bibliographystyle{alpha}%
%\bibliography{TG-articles}
%%------------------------------------------------------------------------------

%------------------------------------------------------------------------------
%------------------------------------------------------------------------------
\end{document}